\documentclass[11pt]{amsart}
\usepackage{amstext}
\usepackage{amsthm,euscript}
\usepackage{amscd}
\usepackage{amssymb}
\usepackage{amsmath}
\usepackage{mathtools}
\usepackage[all]{xy}
\usepackage{paralist}
\usepackage{enumitem}
\usepackage{moreenum}
\usepackage{scalerel,stackengine}
\usepackage[colorlinks=true, pdfstartview=FitV, linkcolor=blue,
citecolor=blue,urlcolor=blue,breaklinks=true]{hyperref}
\usepackage{subfiles} 
\swapnumbers 

\newcommand{\arxiv}[1]{\href{http://arxiv.org/abs/#1}{\tt arXiv:\nolinkurl{#1}}}
\newtheorem{thm}[subsection]{Theorem}

\newtheorem{prop}[subsection]{Proposition}

\newtheorem{lem}[subsection]{Lemma}

\newtheorem{cor}[subsection]{Corollary}

\theoremstyle{definition}

\numberwithin{equation}{subsection}

\newcommand\sm{\smallskip}
\newcommand\ms{\medskip}

\def\limind{\mathop{\oalign{lim\cr\hidewidth$\longrightarrow$\hidewidth \cr}}}
\def\limproj{\mathop{\oalign{lim\cr\hidewidth$\longleftarrow$\hidewidth\cr}}}

\newcommand{\simlgr}{\buildrel \sim \over \longrightarrow}

\newcommand{\fppf}{_\mathrm{fppf}}

\newcommand{\longto}{\longrightarrow}
\newcommand{\we}{\wedge}

\newcommand{\ti}{^\times}
\newcommand{\lan}{\langle}\newcommand{\ran}{\rangle}
\def\co{\colon}
\def\ot{\otimes} 

\def\ideal{\triangleleft}
\newcommand{\me}{^{-1}}
\def\dar[#1]{\ar@<2pt>[#1]\ar@<-2pt>[#1]}

\stackMath
\newcommand\reallywidehat[1]{%
\savestack{\tmpbox}{\stretchto{%
  \scaleto{%
    \scalerel*[\widthof{\ensuremath{#1}}]{\kern.1pt\mathchar"0362\kern.1pt}%
    {\rule{0ex}{\textheight}}
  }{\textheight}%
}{2.4ex}}%
\stackon[-6.9pt]{#1}{\tmpbox}%
}

\newcommand{\rmA}{\mathrm{A}}

\newcommand{\rmB}{\mathrm{B}}

\newcommand{\can}{\operatorname{can}}
\newcommand{\rmD}{\mathrm{D}}

\newcommand{\Dyn}{\mathrm{Dyn}} 

\newcommand{\Hom}{\operatorname{Hom}}

\newcommand{\hyp}{\mathrm{hyp}}

\newcommand{\rmi}{\operatorname{i}}
\newcommand{\Id}{\operatorname{Id}}

\newcommand{\inc}{{\operatorname{inc}}}

\newcommand{\Jac}{{\operatorname{Jac}}}

\newcommand{\Ker}{\operatorname{Ker}}

\newcommand{\Mor}{\operatorname{Mor}}

\newcommand{\rmN}{\operatorname{N}}

\newcommand{\rank}{\operatorname{rank}}

\newcommand{\rad}{\operatorname{rad}}

\newcommand{\res}{\operatorname{res}}

\newcommand{\Spec}{\operatorname{Spec}}
\newcommand{\Specmax}{\operatorname{Specmax}}

\newcommand{\Wq}{\operatorname{W}_q}

\newcommand{\hWq}{\widehat{\operatorname{W}}_q}

\newcommand\sfj{{\sf j}} 

\newcommand\sfq{{\sf q}}
\newcommand\sfR{{\sf R}}\newcommand\sfr{{\sf r}}
 
 \newcommand\sft{{\sf t}} \newcommand\sftr{\sft\sfr}

\newcommand{\Of}{\mathrm{Of}} 

\newcommand{\uG}{\mathbf G}

\newcommand{\uP}{\mathbf{P}}

\newcommand{\uQ}{\mathbf{Q}}

\newcommand{\uSO}{\mathbf {SO}}

\newcommand{\uT}{\mathbf T}

\newcommand{\uX}{\mathbf X}

\newcommand{\m}{\mathfrak m}

\newcommand{\p}{\mathfrak p}

\newcommand{\Ralg}{R\mathchar45\mathbf{alg}}
\newcommand{\Salg}{S\mathchar45\mathbf{alg}}

\newcommand{\GG}{{\mathbb G}}
\newcommand{\HH}{{\mathbb H}}
\newcommand{\MM}{{\mathbb M}}
\newcommand{\NN}{{\mathbb N}}

\newcommand{\PP}{{\mathbb P}}

\newcommand{\ZZ}{{\mathbb Z}}

\newcommand\al{\alpha}
\newcommand\be{\beta}

\newcommand\ka{\kappa}
 
 \newcommand\vphi{\varphi}

 \newcommand\vth{\vartheta}

\begin{document}
\title[Odd Degree Extension Theorem]{Springer's Odd Degree Extension Theorem for Quadratic Forms over Semilocal Rings}
\author[P. Gille]{Philippe Gille}
\address{UMR 5208 du CNRS -
Institut Camille Jordan - Universit\'e Claude Bernard Lyon 1, 43 boulevard du
11 novembre 1918, 69622 Villeurbanne cedex - France }
\thanks{P. Gille was supported by the ANR project \emph{Geolie} ANR-15-CE40-0012}
\email{gille@math.univ-lyon1.fr}

\author[E. Neher]{Erhard Neher}
\address{ Department of Mathematics and Statistics, University of Ottawa,
Ottawa, Ontario, Canada, K1N 6N5}
\thanks{E. Neher acknowledges partial support from NSERC through a Discovery grant}
\thanks{This work was supported by the Labex Milyon (ANR-10-LABX-0070) of Universit\'e de Lyon, within the program ``Investissements d'Avenir'' (ANR-11-IDEX- 0007) operated by the French National Research Agency (ANR)}
\email{neher@uottawa.ca}

\date{\today}
\maketitle
\smallskip
\[ \text{\em In memory of T.A. Springer} \]
\bigskip

\noindent{\bf Abstract:} A fundamental result of Springer says that a quadratic form over a field of characteristic $\ne 2$ is isotropic if it is so after an odd degree field extension. In this paper we generalize Springer's theorem as follows. Let $R$ be  an arbitrary semilocal ring, let $S$ be a finite $R$--algebra of constant odd degree, which is \'etale or generated by one element, and let $q$ be  a nonsingular $R$--quadratic form whose base ring extension $q_S$ is isotropic. We show that then already $q$ is isotropic.
\medskip


\medskip


\bigskip

\section*{Introduction}

A celebrated result of Springer \cite{springer-qf} says that a quadratic form over a field $F$ of characteristic $\ne 2$ that becomes isotropic over an odd degree field extension is already isotropic over $F$. The result was conjectured by Witt and was also proven, but not published by E.~Artin. An account of Springer's result is given in \cite[VII, Thm.~2.7]{Lam-qf} or \cite[II, Thm.~5.3]{Sc}. It is proven for arbitrary fields in \cite[18.5]{EKM}.

Springer's Theorem has many important consequences, for example for Witt groups. It is therefore not surprising that it has been generalized by replacing odd degree extensions by more general field extensions, see for example, \cite[Prop.~5.3]{Ho1}, \cite[Cor.~4.2]{Ho2} or \cite[Lem.~2.8]{Lag}, or by replacing the base field $F$ by more general rings. The main result of this paper goes in the latter direction: \sm

\textbf{Odd Degree Extension Theorem.} (Theorem~\ref{thm_springer}) {\em Let $R$ be a semilocal ring, $M$ a finite projective $R$--module, $q\co M \to R$  a nonsingular quadratic form, and $S$ a finite  $R$--algebra of constant odd degree, which is \'etale or one-generated.  If the base ring extension $q_S$ is isotropic, then already $q$ is isotropic.} \sm

Since the terminology regarding quadratic forms over rings is not standard, see the comparison in  \ref{qfba-sing}\eqref{qfba-term}, let us explain what we mean with a nonsingular quadratic form. Recall that the radical of $(M,q)$ with polar form $b_q$ is $\rad (M,q) = \{ m\in M : q(m) = 0 = b_q(m,M)\}$. We call $q$ {\em nonsingular\/}, if $\rad((M,q)_F) = \{0\}$ for all $R$--fields $F$. An example of such a form is a {\em regular} quadratic form, defined by the condition that its polar form induces an isomorphism between $M$ and its dual space. Nonsingular and regular quadratic forms are the same in case $2\in R\ti= \{ u \in R: u R=R\}$, the units of $R$, or in case the projective $R$--module $M$ has constant even rank, but not in general. For example, for $u\in R\ti$ the quadratic form $R \to R$, $x \mapsto ux^2$ is always nonsingular, but regular only if $2\in R\ti$. We call a finite projective $R$--algebra $S$ of constant degree $d$  {\em one-generated\/} if it is generated by one element, equivalently, $S\cong R[X]/f$ for a monic polynomial of degree $d$.

Special cases of the Odd Degree Extension Theorem have been proven before; all of them assumed $2\in R\ti$, so that nonsingular = regular. Specifically, it was proven by Panin and Rehmann \cite{panin-rehmann} under the assumption that $R$ is a noetherian local integral domain, which has an infinite residue field, and that $S$ is \'etale (necessarily one-generated in this case). It was extended in \cite{panin-pimenov} by Panin-Pimenov to $R$ a semi-local noetherian integral domain whose residue fields are all infinite. It was stated in \cite{Scully} for $R$ semilocal, $2\in R\ti$ and $S$ finite \'etale of constant odd degree. However,  there is a serious gap in the proof of the crucial Proposition 3.1 of \cite{Scully}. As a result, the Odd Degree Extension Theorem, the main theorem of \cite{Scully},  is not proven in the generality stated.  \sm

For proving the Odd Degree Extension Theorem one can easily reduce to $M$ of constant rank $r\ge 2$. The case $r=2$ (Lemma~\ref{ratwo}) holds for any finite $R$--algebra; our proof uses a consequence of Deligne's trace homomorphism for the flat cohomology of abelian affine group schemes (Lemma~\ref{del-tr}). For $r \ge 3$ we follow an approach inspired by the paper \cite{panin-rehmann} of Panin and Rehmann and first deal with one-generated algebras. A crucial step here is
proving that isotropic elements can be lifted from $M \ot_R (S/\Jac(R)S)$ to $M\ot _R S$ (Corollary~\ref{cor_qquadric}). We obtain this as an application of Demazure's Conjugacy Theorem for reductive $R$--group schemes. We should point out, that in this part we use the classical version of the Odd Degree Extension Theorem ($R$ a field, \cite{EKM}). Once the case of one-generated algebras has been settled, the case of \'etale algebras easily follows by applying a recent result of Bayer-Fluckiger, First and Parimala \cite{BFP} on embeddings of finite \'etale $R$--algebras into one-generated $R$--algebras. \sm

{\em Structure of the paper.} Section~\ref{sec:qf} presents a review of nonsingular quadratic forms and proves some results needed for our proof of the Odd Degree Extension Theorem in the first part of section~\ref{sec:soet}. In the second part of that section we prove some consequences of the theorem well-known in case of base fields. The paper closes with two appendices, \ref{app:lifting} and \ref{app:torEN}. Their results are used in our proof of the Odd Degree Extension Theorem, but they are of interest beyond that Theorem. \sm

{\em Basic notation.}
Throughout, $R$ is a commutative associative unital ring. We do not assume that $2 \in R\ti$, unless this is explicitly so stated. We denote by $\Specmax(R)$ the subspace of the spectrum $\Spec(R)$ whose elements are the maximal ideals of $R$.
Modules, quadratic forms and algebras will be defined over $R$. We use $\Ralg$ to denote the category of commutative associative unital $R$--algebras. Objects of $\Ralg$ are sometimes also referred to as $R$--rings. If $M$ is an $R$--module and $S \in \Ralg$, we put $M_S = M \ot_R S$.

We call a submodule $U$ of an $R$--module $M$ complemented if there exists a submodule $U'$ of $M$ such that $U \oplus U' = M$. We say an $R$--module $M$ is finite projective if it is finitely generated projective (= locally free of finite rank). A finite $R$--algebra is an algebra $S\in \Ralg$ whose underlying $R$--module is finitely generated, but not necessarily projective. A finite projective $R$--algebra is an $R$--ring $S$ whose underlying $R$--module is finite projective. The norm of such an $R$--algebra $S$ is denoted $\rmN_{S/R}$. A finite projective $R$--algebra $S$ is called finite \'etale if it is separable, see for example \cite{Ford} or \cite{K} for other equivalent definitions.
We call $S\in \Ralg$ an \'etale $R$--algebra of degree $d\in \NN$ if $S$ is a finite \'etale algebra whose underlying $R$--module is projective of constant rank $d$.

\section{Quadratic forms}\label{sec:qf}

This section starts with a review of known facts on bilinear and quadratic forms over arbitrary rings, thereby also establishing our notation and terminology, \ref{qfba}--\ref{meta}. We study quadratic forms over semilocal rings in the second part of this section.

\subsection{Quadratic and symmetric bilinear forms (terminology, basic facts)}\label{qfba}
\begin{inparaenum}[(a)]
\item \label{qfba-aa} A {\em (symmetric) bilinear module\/} is a pair $(M,b)$ consisting of a finite projective $R$--module and a symmetric $R$--bilinear map $b\co M \times M \to R$. Except when otherwise stipulated,  we will only consider symmetric bilinear forms. It is therefore convenient to simply speak of bilinear forms and bilinear modules. We use the symbol $(M,b) \cong (M',b')$ to denote isometric bilinear modules. When $M$ is clear from the context or unimportant, we will sometimes write $b$ for $(M,b)$.
    Given a bilinear module $(M,b)$, its {\em adjoint\/} is the $R$--linear map $b^* \co M \to M^*= \Hom_R(M,R)$, $m \mapsto b(m, \cdot)$. We call $(M,b)$ {\em regular\/} if $b^*$ is an isomorphism.

   \sm

   \item \label{qfba-aaq} A {\em quadratic module\/} is  a pair $(M,q)$ where $M$ is a finite projective $R$--module and $q\co M \to R$ is an $R$--quadratic form. We use $b_q$ to denote the polar form of $q$.
       We call $q$ or $(M,q)$ {\em regular\/} if $b_q$ is regular. As for bilinear modules, $(M,q) \cong (M',q')$ indicates isometric quadratic modules. We sometimes write $q$ instead of $(M,q)$ if $M$ is unimportant or clear from the context. \sm

\item\label{qfba-a} ({\em Base change}) Let $S\in \Ralg$, and let $(M,b)$ be an $R$--bilinear module. There exists a unique $S$--bilinear form $b_S \co M_S \times M_S \to S$ satisfying $b_S(m_1 \ot s_1, m_2 \ot s_2) \allowbreak = b(m_1, m_2) s_1s_2$ for $m_i \in M$ and $s_i \in S$. Given an $R$--quadratic module $(N,q)$, there exists a unique $S$--quadratic form $q_S \co N_S \to S$ satisfying $q_S(n \ot s) = q(n)s^2$ for $s\in S$ and $n\in N$. The polar of $q_S$ is the base change of the polar of $q$.     If $(M, b)$ (or $(N,q)$) is regular, then so is $(M_S, b_S)$ (respectively $(N_S, q_S)$). 
    \sm

\item\label{qfba-tens} ({\em Tensor products}) Let $(M_i, b_i)$, $i=1,2$, be  bilinear modules. There exists a unique symmetric bilinear form $b_1 \ot b_2$ on $M_1 \ot_R M_2$ satisfying $(b_1 \ot b_2)(m_1 \ot m_2, m_1'\ot m_2') = b_1(m_1, m'_1)\, b_2(m_2, m_2')$ for $m_i, m_i' \in M_i$.
Given an $R$--bilinear module $(M,b)$ and an $R$--quadratic form $(N,q)$, there exists a unique $R$--quadratic form $b\ot q \co M \ot_R N\to R$ satisfying
\begin{equation}\label{tenssq0}
\begin{split} (b \ot q)\, (m \ot n) &= b(m,m) \, q(n), \; \text{and}\\
  b_{b\ot q}(m \ot n, m' \ot n') &= b(m, m')\, b_q(n, n')
\end{split}
\end{equation}
for all $m, m'\in M$ and $n, n'\in N$, where $b_q$ is the polar form of $q$. It is called the {\em tensor product of $(M,b)$ and $(N,q)$}, \cite[Thm.~1]{Sah}. The polar form of $b \ot q$ is the tensor product of the symmetric bilinear forms $b \ot b_q$. If $(M,b)$ and $(N,q)$ are regular, then so is $(M,b) \ot (N,q)$.
The tensor product is compatible with base change: for $S\in \Ralg$ we have
\begin{equation*}
(M,b)_S \ot_S (N,q)_S \simlgr \big((M,b) \ot_R (N,q)\big)_S
\end{equation*}%
with respect to $m\ot s_1 \ot n \ot s_2 \mapsto m\ot m \ot s_1 s_2$.
\sm

\item \label{qfba-or} Given a bilinear module $(M,b)$, the {\em orthogonal module\/} of a  submodule $U\subset M$ is $U^\perp = \{m \in M : b(m, U) = 0 \}$. We call $U$ {\em totally isotropic\/} if $U \subset U^\perp$. A submodule $U$ of a quadratic module $(M,q)$ is {\em totally isotropic\/} if $q(U) = 0$, in which case $U$ is also a totally isotropic submodule of $(M,b_q)$.
      If $(M,b)$ is a bilinear module and $U\subset M$ is a submodule such that $(U, b|_U)$ is regular, then $M = U \oplus U^\perp$.
\end{inparaenum}

\subsection{Nonsingular quadratic forms.}\label{qfba-sing} Recall that the {\em radical\/} of a quadratic module $(M,q)$ is  $\rad (M,q) = \{ m\in M : q(m) = 0 = b_q(m,M)\}$.  It is a submodule of $M$ and satisfies $(\rad(M,q))_S \subset \rad( (M,q)_S)$, $S\in \Ralg$. We call $(M,q)$ or simply $q$ {\em nonsingular\/} if $\rad((M,q)_F) = \{0\}$ for all fields $F\in \Ralg$. A {\em quadratic space\/} is a quadratic module $(M,q)$ with a nonsingular $q$. We will use the following properties of nonsingular forms (for details see the references in  \eqref{qfba-term}).

\begin{inparaenum}[(a)]
\item\label{qfba-term}  ({\em Comparison of terminology}) Our terminology of a regular bilinear form or regular quadratic form follows \cite{K} and \cite{Sc} (except that in these references ``regular'' and ``nonsingular'' are used interchangeably), but a  regular bilinear form as defined here is called ``non singular'' in \cite{Ba} and  ``nondegenerate'' in \cite{EKM}.
    A nonsingular quadratic form is called ``non-degenerate'' in \cite{Co1}, ``nondegenerate'' in \cite{EKM}, ``semiregular'' in \cite{K} in case of constant odd rank, cf.\ \eqref{qfba-sing-v}, and ``separable'' in \cite{P-Fields}.
    \sm

\item A regular quadratic form is nonsingular, since $\rad((M,q)_S) = 0$ for a regular form and any $S\in \Ralg$ by \ref{qfba}\eqref{qfba-a}. \sm

  \item \label{qfba-sing-ii} If $2\in R\ti$, a nonsingular form is regular. Indeed, if $2\in R\ti$, then $\rad(q_S) = \{m\in M_S : b_{q_S}(m,M_S)=0\}$, $S\in\Ralg$, and so the adjoint of $b_q$ is an isomorphism by Nakayama. 
      \sm

  \item \label{qfba-sing-iii}  If $q$ is nonsingular, then so is $q_S$ for any $S\in \Ralg$.  %
     \sm

   \item \label{qfba-sing-v} Let $(M,q)$ be a quadratic module with $M$ of constant even rank. Then
      $q$ is  nonsingular if and only if $q$ is regular, cf.\ \eqref{qfba-one} below. \sm

 \item \label{qfba-sing-vi} A quadratic $R$--space $(M,q)$ with $M$ faithfully projective is {\em primitive\/} in the sense that $q(M)$ generates $R$ as ideal. If $R$ is semilocal, then $q(m) \in R\ti$ for some $m\in M$ (which is unimodular in the sense of \ref{isotrop}).
\sm

  \item \label{qfba-sing-vii} Let $(M,q) = (M_1, q_1) \perp (M_2, q_2)$ be a direct sum of quadratic modules. If $q$ is nonsingular, then so are $q_1$ and $q_2$. Conversely, if $q_1$ is regular, then $q$ is nonsingular if and only if $q_2$ is nonsingular.
      If $q_1$ and $q_2$ are nonsingular, then $q$ need not be nonsingular, see the example in \ref{carhyp-cor}.
\sm

\item \label{qfba-one} ({\em $1$--dimensional forms}) Let $u\in R\ti$. We define $\lan u \ran_b \co R \times R \to R$, $(r_1, r_2) \mapsto u r_1 r_2$ and $\lan u \ran_q \co R \to R$, $r \mapsto ur^2$. Then $\lan u \ran_b $ is a regular bilinear form. The polar of the quadratic form $\lan u \ran_q$ is $2 \lan u \ran$, whence $\lan u \ran_q$ is regular if and only if $2\in R\ti$, but $\lan u \ran_q$ is always nonsingular.
    We abbreviate $\lan u \ran = \lan u \ran_b$  if the meaning of $\lan u \ran$ is clear from the context.

   If $(M,b)$ and $(N,q)$ are bilinear and quadratic modules respectively, then
\begin{equation*} \label{qfba-tens-2}
 \lan 1\ran_b \ot b \cong  b, \quad \lan 1 \ran_b \ot q \cong q
 \end{equation*}
under the standard isomorphism $R \ot_R  M \simlgr M$. However,
$ b \ot \lan 1\ran_q = q_b$, where $q_b \co M \to R$, $m \mapsto b(m,m)$, is the quadratic form associated with $b$. Its polar is $b_{q_b} = 2b$.
\sm

\item\label{qfba-red} ({\em Reduction to constant rank}) Let $R=R_0 \times \cdots \times R_n$ be a direct product of rings. For $0 \le i \le n$ we view $e_i = 1_{R_i}$ as an idempotent of $R$. Thus $(e_0, \ldots, e_n)$ is a system of mutually orthogonal idempotents  with $R_i=Re_i$. Any quadratic $R$--module $(M,q)$ uniquely decomposes into the orthogonal sum of quadratic $R$--modules
    \begin{equation}\label{qfba-red1}
     (M,q) = (M_0, q_0) \perp \cdots \perp (M_n,q_n),
    \end{equation}
     where $M_i = M e_i$ and $q_i$ is restriction of $q$ to $M_i$. By a slight abuse of notation we will also consider $(M_i, q_i)$ as a quadratic $R_i$--module and identify $(M,q)$ with the direct product $(M_1, q_1) \times \cdots \times (M_n, q_n)$ of quadratic $R_i$--modules. Conversely, given quadratic $R_i$--modules  $(M_i, q_i)$, $0 \le i \le n$, we obtain a quadratic $R$--module $(M,q)$ for which \eqref{qfba-red1} holds. The quadratic $R$--module $(M,q)$ is regular (resp.\ nonsingular) if and only if all quadratic $R_i$--modules $(M_i, q_i)$ are regular (resp.\ nonsingular).

    For arbitrary $R$, a finite projective $R$--module $M$ gives rise to a decomposition $R=R_0 \times \cdots \times R_n$ and hence to $M = M_0 \times \cdots \times M_n$ such that $M_i$ is a projective $R_i$--module of constant rank $i$. The discussion above then describes the reduction of quadratic modules to quadratic modules of constant rank. \sm

\item\label{qfba-wi} ({\em Witt cancellation})
Let $R$ be a semilocal ring, let $q_1$, $q_2$, $q_1'$ and $q_2'$ be quadratic forms where $q_1$ is regular and $q_2$ is nonsingular and of positive rank.  If $q_1 \cong q_1'$ and $q_1 \perp q_2 \cong q_1' \perp q_2'$, then $q_2 \cong q_2'$. Indeed, since the nonsingular form $q_2$ is primitive by \eqref{qfba-sing-vi}, this follows from \cite[K\"urzungssatz]{Knebusch}. Witt cancellation with all forms $q_i$ and $q_i'$ being regular, is proven in \cite[III, (4.3)]{Ba}.
\end{inparaenum}

\subsection{Metabolic and hyperbolic spaces.} \label{meta}
Let $(U,b)$ be a bilinear module.  The {\em metabolic space\/} associated with $(U,b)$ is the bilinear module $\MM(U,b) = (U \oplus U^*, b_{\MM(U,b)})$ whose bilinear form $b_{\MM(U,b)}$  is defined by
\begin{equation*} \label{meta-1}
 b_{\MM(U,b)} (u + \vphi, v + \psi) = b(u,v) + \vphi(v) + \psi(u).
 \end{equation*}
It is a regular bilinear form, justifying the terminology ``space''.
By definition, the {\em hyperbolic bilinear space\/} is $\HH_b(U) = \MM(U, 0)$ where $0$ is the null form.

Given a finite  projective $R$--module $U$, the associated {\em hyperbolic space $\HH(U)$} is the quadratic module $(U^* \oplus U, \hyp)$ with $\hyp(\vphi \oplus u) = \vphi(u)$.  The polar form of the hyperbolic quadratic form $\hyp$ is the bilinear form $b_{\HH_b(U,0)}$, in particular, $\hyp$ is nonsingular. We call $\HH(R)$ the {\em hyperbolic plane\/}.

We say a bilinear module $(M, b)$ is {\em metabolic\/} if there exists a bilinear module $(U, b_U)$ such that $(M, b)\cong \MM(U,b_U)$. The same terminology will be applied to hyperbolic quadratic spaces. The following facts are for example proven in \cite[I, \S3 and \S4]{Ba}.
\sm

\begin{inparaenum}[(a)]
\item \label{meta-c} %
 Let $(M, b)$ be a regular bilinear module and let $U \subset M$ be a totally isotropic complemented submodule. Then there exists a submodule $V\subset M$ such that $U \cap V = 0$,  $(U\oplus V, b|_{U \oplus V})$ is metabolic and hence $M = (U \oplus V) \perp (U \oplus V)^\perp$. If $b(m_1, m_2) = b_0(m_1, m_2) + b_0(m_2, m_1)$ for some in general non-symmetric bilinear form $b_0 \co M \times M \to R$, one can choose $V$ such that $(U\oplus V, b|_{U \oplus V})$ is a hyperbolic bilinear space.
\sm

\item\label{meta-b} %
({\em Characterization of metabolic spaces}) A regular bilinear module $(M, b)$ is metabolic if and only if one of the following conditions holds:
\end{inparaenum} \begin{enumerate}[label={\rm (\roman*)}]

  \item\label{meta-bi}  $M$ contains a totally isotropic complemented submodule $V$ with $V = V^\perp$, a so-called {\em Lagrangian},

 \item \label{meta-bii} $M$ contains a totally isotropic and complemented submodule $U$ satisfying $\rank_\p M = 2 \rank_\p U$ for all $\p \in \Spec(R)$.
\end{enumerate}
In this case $(M,b) \cong \MM(U)$ for $U \cong V^*$. In particular, $\MM(U)$ is free whenever $V$ is free, and any decomposition $V=V_1 \oplus \cdots \oplus V_n$ gives rise to a decomposition  $\MM(U) = \MM(U_1) \perp \cdots \perp \MM(U_n)$ with $U_i \cong V_i^*$. \sm

To see that $M$ is metabolic in case \ref{meta-bii}, apply \eqref{meta-c} to get $M = (U \oplus V)\perp (U \oplus V)^\perp$ where $U \oplus V$ is metabolic and $\rank_\p = 2 \rank_\p U = \rank_\p M$, whence $(U \oplus V)^\perp = 0$. The last claim follows from $V^* = V_1^* \oplus \cdots \oplus V_n^*$.
\ms

In the remainder of this section we consider quadratic spaces $(M,q)$ over $R$ in two settings: (a) $q$ is regular, and (b) $R$ is semilocal. The results in case (a) are proven in \cite[\S3 and \S4]{Ba}. Since our techniques in case (b) easily also lead to proofs in the setting of (a), we prove (a) and (b) at the same time.  We first recall a folklore result about lifting of subspaces in the semilocal setting.
It is straightforward to prove or can be obtained by specializing the Reduction Theorem \cite[II, (4.6.1)]{K}.%

\begin{lem}\label{lift-mod} Let $R$ be a semilocal ring, let $M$ be a finite  projective $R$--module and let $U\subset M$ be a complemented submodule. For a maximal ideal $\m\ideal R$ we put $\ka(\m) = R/\m$, $M(\m) = M \ot_R \ka(\m) = M / \m M$, and analogously for $U(\m)$. We further assume that $r\in \NN_+$ and that for every $\m \in \Specmax(R)$ there exists an $r$--dimensional subspace $W[\m] \subset M(\m)$ with $U(\m) \cap W[\m] = \{0\}$.

Then there exists a free submodule $W\subset M$ of rank $r$ which satisfies $W(\m) = W[\m]$, $U \cap W = \{0\}$ and which has the property that $U \oplus W$ is complemented in $M$. \end{lem}

\begin{prop}\label{lift-qf} Let $(M,q)$ be a quadratic space over $R$, let $U \subset M$ be a complemented totally isotropic submodule, and assume one of the following. \begin{enumerate}[label={\rm (\alph*)}]

\item\label{lift-qfa} $q$ is regular, or

\item \label{lift-qfr} $R$ is semilocal.

\end{enumerate}
Then there exists a totally isotropic submodule $V \subset M$ such that $U \cap V = \{0\}$, $(U \oplus V, q_{U \oplus V}) \cong \HH(U)$ and hence $M = (U\oplus V) \oplus (U \oplus V)^\perp$.
\end{prop}

\begin{proof} \begin{inparaenum}[(I)]
\item \label{lift-qf-II} ({\em Intermediate step}) Suppose there exists a submodule $W \subset M$ satisfying
\begin{equation}\label{lift-qf-1} \begin{split}
 & U \cap W = \{0\}, \text{ and for which the canonical map  } \\
  &\be \co U \simlgr W^*, \quad u \mapsto (w \mapsto b_q(u,w))
\end{split} \end{equation}
is an isomorphism of $R$--modules. By \cite[I, (1.7)]{Ba}, we can choose a not necessarily symmetric bilinear form $b_0$ satisfying $b_0(m,m) = q(m)$ for all $m\in M$. By \eqref{lift-qf-1}, for every $w\in W$ there exists a unique $u_w \in U$ such that $b_q(u_w, w') = b_0(w, w')$ holds for all $w'\in W$. Because of uniqueness of the $u_w$, the map $W\to U$, $w \mapsto u_w$, is $R$--linear. Then
$V = \{ w - u_w : w \in W\}\cong W$ is a totally isotropic submodule: since $q(u_w)=0$ we have
\[ q(w - u_w) = q(w) - b_q(w, u_w) = b_0(w,w) - b_q(u_w, w) = 0.
\]
 Moreover, $U \cap V = 0$ and the canonical map $U \simlgr V^*$, $u \mapsto (v \mapsto b_q(u,v))$ is an isomorphism. Hence $(U\oplus V, q|_{U \oplus V}) \cong \HH(U)$. Since $\HH(U)$ is a regular quadratic module, \ref{qfba}\eqref{qfba-or} applies and yields 
  $M = (U \oplus V) \oplus (U \oplus V)^\perp$.

In the remainder of the proof we will establish the existence of a submodule $W$ satisfying \eqref{lift-qf-1}. We point out that \eqref{lift-qf-II} applies to the two cases \ref{lift-qfa} and \ref{lift-qfr}. \sm

\item \label{lift-qf-III} ({\em Case {\rm \ref{lift-qfa}} in general and case {\rm \ref{lift-qfr}} with $R$ a field}) Let $U^\perp = \{m \in M: b_q(m, U) = 0 \}$. We have a well-defined pairing
    \[
    \al \co U \times M/(U^\perp) \to R, \quad (u, \bar m) \mapsto b_q(u, m)
    \]
    which is regular: if $\al(u, \bar m) = b_q(u, m) = 0$ for all $m\in M$, then
    $u\in \rad(b_q)$. Hence $u=0$ in case \ref{lift-qfa}, while in case \ref{lift-qfr} we get $u\in \rad(q)$ because $q(u) = 0$, so that again $u=0$ follows, using that $\rad(q) = 0$ if $R$ is a field. Also, if $\al(U, \bar m) = b_q(U, m) = 0$, then $m\in U^\perp$, and therefore $\bar m = 0$ follows. We now get that $\al$ induces an isomorphism $\al^* \co U \simlgr (M/ U^\perp)^*$. We claim that there exists a submodule $W \subset M$ such that $M = U^\perp \oplus W$, and that therefore we have the canonical isomorphism $\can \co W\simlgr M/U^\perp$. The existence of $W$ follows in case \ref{lift-qfa} from \cite[I, (3.2)(a)]{Ba}, saying that $U^\perp$ is complemented in $M$. It is obvious in case \ref{lift-qfr} with $R$ a field. Denoting by $\can^* \co (M/U^\perp)^* \simlgr W^*$ the dual of the isomorphism $\can$, we have $\big( ( \can^* \circ \al^*)(u)\big)(w) = \al^*(u)\big( \can (w)\big) = \al^*(u)(\bar w) = b_q(u,w) = (\be(u)\big)(w)$, i.e., $W$ satisfies \eqref{lift-qf-1}. Indeed, we have $W \cap U = \{0\}$ (since $U \subset U^\perp$ and $W \cap U^\perp = \{0\}$) and with $\be = \can^* \, \circ\, \al^*$. Now \eqref{lift-qf-II} finishes the proof in case \ref{lift-qfa}, and in case \ref{lift-qfr} with $R$ a field.
    Before we can deal with a semilocal $R$ in case \ref{lift-qfr} we make a further reduction. \sm

\item \label{lift-qf-I} ({\em Reduction to constant rank}) Since $U$ is complemented, it is finitely generated projective. By \ref{qfba}\eqref{qfba-red}, we can therefore decompose $R= R_1 \times \cdots \times R_n$ and correspondingly
\[ (M,q) =  (M_1, q_1) \perp \cdots \perp (M_n, q_n), \quad
   U = U_1 \times \cdots \times U_n \]
such that each $U_i\subset M_i$ is a complemented submodule of constant rank and
a totally  isotropic submodule of the $R_i$--quadratic space $(M_i, q_i)$. If $V_i\subset M_i$, $1\le i \le n$, are submodules as in the claim of the lemma, then $V= V_1 \times \cdots \times V_n$ satisfies the conditions for $(M,q)$. Without loss of generality we can therefore assume that $U$ has constant rank, say rank $r$. It 
is then free of rank $r$.
\sm

\item ({\em Case \ref{lift-qfr} in general})  By \eqref{lift-qf-I} we can assume that $U$ is free of rank $r$. Using the notation of Lemma~\ref{lift-mod}, we know that $q_{\ka(\m)}$ is nonsingular by \ref{qfba-sing}\eqref{qfba-sing-iii}. Thus, by \eqref{lift-qf-III}, the lemma holds for $\ka(\m)$. We can therefore choose an $r$--dimensional subspace $W[\m]$ such that $U(\m) \cap W[\m] = \{0\}$ and $U(\m) \simlgr W[\m]^*$ via $b_{\ka(\m)}$. By Lemma~\ref{lift-mod}, the $W[\m]$ lift to a submodule $W$ satisfying $U \cap W = \{0\}$ and $W(\m)= W[\m]$. Moreover, the map $U \longto W^*$, induced by $b_q$, is an isomorphism by Nakayama, 
    since it is an isomorphism after passing to each $\ka(\m)$. Again \eqref{lift-qf-II} finishes the proof.
\end{inparaenum}
\end{proof}

\begin{cor}[\bf Characterization of hyperbolic spaces] \label{carhyp}
Let $(M,q)$ be a quadratic $R$--space and assume that $q$ is regular or that $R$ is semilocal. Then the following are equivalent: \sm

\begin{enumerate}[label={\rm (\roman*)}]
  \item \label{carhyp-i} $(M,q)$ is hyperbolic;

  \item \label{carhyp-ii} $M$ admits a complemented submodule $U$ satisfying $q(U) = 0$
  and $2 \rank_\p U = \rank_\p M$ for all $\p \in \Spec(R)$;

  \item \label{carhyp-iii} $M$  admits a complemented submodule $U$ satisfying $q(U) = 0$ and $U= U^\perp$.
\end{enumerate}
In this case $(M,q) \cong \HH(U)$.
\end{cor}

\begin{proof}
  \ref{carhyp-i} $\implies$ \ref{carhyp-ii} being obvious because $\rank_\p U = \rank_\p U^*$, let us assume \ref{carhyp-ii} and prove \ref{carhyp-iii}. By Proposition~\ref{lift-qf},  there exists a submodule $V \subset M$ such that $(U \oplus V, q|_{U \oplus V})\cong \HH(U)$ is hyperbolic and $M = (U \oplus V) \oplus (U \oplus V)^\perp$. Since $V \cong U^*$ as $R$--modules, $\rank_\p (U \oplus V) =  \rank_\p M$, whence $U \oplus V = M$ and $U^\perp = U \oplus (U^\perp \cap V)$ with $U^\perp \cap V \cong  \{\vphi \in U^* : \vphi(U) = 0 \} = \{0\}$.
The proof of \ref{carhyp-iii} $\implies$ \ref{carhyp-i} follows the same pattern.  \end{proof}

\begin{cor}\label{carhyp-cor} Let $(M,q)$ and $(M',q')$ be regular quadratic modules. \sm

\begin{inparaenum}[\rm (a)] \item {\rm \cite[I, (4.7.i)]{Ba}} \label{carhyp-cor-a} If $(M,q)$ and $(M', q')$ are isometric, then the quadratic module $(M , q) \perp (M', -q')$ is hyperbolic: $(M,q) \perp (M',-q') \cong \HH(M)$.\sm

\item\label{carhyp-cor-b} Conversely, if $R$ is semilocal and if $(M,q) \perp (M',-q') \cong \HH(M)$, then $(M,q) \cong (M',q')$.
\end{inparaenum}
\end{cor}

\begin{proof} \eqref{carhyp-cor-a} Let $f \co (M,q) \to (M',q')$ be an isometry. The quadratic form $q\perp (-q')$ is regular. The diagonal submodule $U = \{\big(m,f(m)\big) : m \in M\}\subset M \oplus M'$ is complemented by $\{ (m,0): m \in M\}$ and satisfies \ref{carhyp}\ref{carhyp-ii}. Hence $(M,q) \perp (M',-q') \cong \HH(M)$.

\eqref{carhyp-cor-b} By \eqref{carhyp-cor-a}, $(M,q) \perp (M,-q) \cong \HH(M) \cong (M,q) \perp (M',-q')$. Hence $(M,-q) \cong (M'-q')$ by Witt cancellation \ref{qfba-sing}\eqref{qfba-wi}, which implies our claim. \end{proof}
\sm

Corollary~\ref{carhyp-cor}\eqref{carhyp-cor-a} is not true for nonsingular quadratic forms, even over fields. For example, let $(M,q) = (F, \lan u\ran_q) = (M', q')$ with $F$ a field of characteristic $2$ and $u\in F\ti$. Then $(M,q)$ is nonsingular by \ref{qfba-sing}\eqref{qfba-one}, but $0 \ne (1_F, 1_F) \in \rad(q \perp (-q))$, so that $q\perp (-q)$ is singular, hence in particular not hyperbolic.

\subsection{Unimodular and isotropic vectors} \label{isotrop} Let $M$ be a finite  projective $R$--module. For $x\in M$ and $\p \in \Spec(R)$ we put $x(\p) = x \ot_R 1_{\ka(p)}$. Recall that $u\in M$ is called {\em unimodular\/} if $u$ satisfies one of the following equivalent conditions, see e.g.\  \cite[0.3]{Loos}: \begin{enumerate}[label={\rm (\roman*)}]
  \item $Ru$ is complemented and free of rank $1$,

  \item there exists $\vphi \in M^*$ satisfying $\vphi(u) = 1$,

  \item $u(\p) \ne  0$ for all $\p \in \Spec(R)$,

  \item $u(\m) \ne 0$ for all maximal $\m \in \Spec(R)$.

\end{enumerate}

Let $(M,q)$ be a quadratic module. We call $m\in M$ {\em isotropic\/} if $m$ is unimodular and $q(m) = 0$. We say $(M,q)$ is {\em isotropic\/} if $M$ contains an isotropic vector. We note some useful facts.
\sm

\begin{inparaenum}[(a)] \item\label{isotrop-a} If $m$ is a unimodular (resp.\ isotropic) vector of $(M,q)$, then $m \ot 1_S$ is a unimodular (resp.\ isotropic)   vector of $(M,q)_S$ for any $S\in \Ralg$.
\sm

\item\label{isotrop-suff} Let $R[X]$ be the polynomial ring over $R$ in the variable $X$ and let $(M,q)$ be a quadratic $R$--module. For $v=v(X) \in M \ot_R R[X]$ define the affine $R$--scheme
    \[ Z_v = \{x\in \GG_{a,R} : v(x) = 0 \},  \]
    whose $T$--points, $T\in \Ralg$, is the set $Z_v(T) = \{ t\in T: v(t)=0\}$ where $v(t) \in M \ot_R T$ is obtained by substituting $t$ for $X$.  Then
    \begin{equation} \label{isotrop-suff1}  \text{$Z_v$ empty}\quad \implies \quad \text{$v$ unimodular.}
     \end{equation}
     Indeed, if $\p \in \Spec(R[X])$, then $v(\p) = v(X \ot 1_{\ka(\p)}) \ne 0$ since otherwise $X\ot 1_{\ka(\p)} \in Z_v(\ka(\p))$.  \sm

\item \label{isotrop-c} Let $(M,q)$ be a quadratic space over $R$ and assume that $q$ is regular or that $R$ is semilocal. By Proposition~\ref{lift-qf},  any isotropic vector embeds into a hyperbolic plane $\HH = \HH(R)$ and $(M,q) = \HH \perp (M_1, q_1)$. In particular, $\rank_\p M \ge 2$ holds for every $\p \in \Spec(R)$. \sm

 \item\label{isotrop-d} Let $M$ be a projective $R$--module of constant rank $2$ and let $q\co M \to R$ be a nonsingular quadratic form. By \ref{qfba-sing}\eqref{qfba-sing-v}, $q$ is nonsingular  if and only if $q$ is regular. Hence,  by \eqref{isotrop-c}, we have the implication ``$\implies$'' of
     \begin{equation} \label{isotrop-d1}
     \text{$(M,q)$ isotropic} \quad \iff \quad (M,q) \cong \HH(R).
     \end{equation}
     Thus, in this case $M$ is free of rank $2$. The other direction in \eqref{isotrop-d1} holds because in $\HH(R) = R \oplus R$ with the hyperbolic form $\hyp$, given by $\hyp(r_1,r_2)= r_1 r_2$, the vector $(1,0)$ is isotropic.
\end{inparaenum}

\section{Springer's odd extension theorem}\label{sec:soet}

Let $(M,q)$ be a quadratic module. To simplify notation, we will often abbreviate
$q(x) = q_S(x)$ for $x\in M_S$ if $S$ is clear from the context. We will also say that {\em $q$ is $S$--isotropic\/} if $q_S$ is isotropic, cf.\ \ref{isotrop}.
 We recall that a quadratic space $(M,q)$ is a quadratic module with a nonsingular $q$.

\begin{thm}[Springer's Theorem] \label{thm_springer} Let $R$ be a semilocal ring and let $(M,q)$ be a quadratic space.  Let $S$ be a finite  $R$--algebra of constant odd degree, which is \'etale or one-generated.  If $q$ is $S$--isotropic, then  $q$ is $R$--isotropic.  \end{thm}

\textbf{Remarks.} The formulation of Theorem~\ref{thm_springer} captures the essence of our work. It is immediate that the theorem extends to $S\in \Ralg$ admitting a tower $R=S_0 \subset S_1  \subset \cdots \subset S_n = S$
for which each $S_i$ is a finite $S_{i-1}$--algebra of constant odd degree which is one-generated or \'etale. The same observation applies to the corollaries stated below. We will use this observation in the proof of case~\ref{prop_odd-b} in Corollary~\ref{prop_odd}.
\sm

The proof of this theorem will be given in \ref{thm_springer-pr}.  Lemma~\ref{ratwo} proves Springer's Theorem in the case of $\rank M =2$. It involves a much weaker condition than $R$ being semilocal or that $S$ is  one-generated or \'etale. By \ref{qfba-sing}\eqref{qfba-sing-v}, a quadratic form on such an $R$--module $M$ is nonsingular if and only if it is regular.

\begin{lem}[Rank $2$]\label{ratwo}
  Let $S\in \Ralg$ be a finite $R$--algebra of constant odd degree and let $(M,q)$ be a quadratic space of constant rank $2$. Then $(M,q)$ is isotropic if and only if $(M,q)_S$ is isotropic.
\end{lem}

\begin{proof}
  It is clear that $(M,q)$ isotropic $\implies (M,q)_S$ isotropic. Let us therefore assume that $(M,q)_S$ is isotropic. Let $A$ be the even part of the Clifford algebra $C=C(q)$ associated with $(M,q)$; we can and will identify $M$ with the odd part of $C$. It is known \cite[V, (2.1)]{K} that $A$ is a quadratic \'etale $R$--algebra; let $n_A$ be its norm.  The multiplication of $C$ makes $M$ a projective right $A$--module of rank $1$ satisfying $q(ma) = q(m) n_A(a)$ for $m\in M$ and $a\in A$.

  Since $q_S$ is isotropic, $(M,q)_S$ is a hyperbolic $S$--plane by \eqref{isotrop-d1}. Hence $C\ot_R S \cong C(q_S)$ is the split quaternion $S$--algebra and $A_S$ is the split quadratic \'etale $S$--algebra. By \cite[III, (4.1.2)]{K}, the automorphism group scheme of the $R$--algebra $A$ is the abelian constant $R$--group scheme $\ZZ/2\ZZ$. Since $S$ has constant odd degree, we are in the setting of the Example in \ref{lem_ttrace}. Thus already $(A, n_A)$ is split, i.e., $A=R \times R$ and $n_A$ is hyperbolic. Let $e_1 = (1_R, 0)$ and $e_2=(0, 1_R)$ be the standard idempotents of $A$. Then $M=Me_1 \times Me_2$ is the direct product of projective $R$--modules of rank $1$ satisfying $q(Me_i) = q(M)n_A(e_i) = 0$ for $i=1,2$. Finally, by Corollary~\ref{carhyp}, $(M,q)$ is hyperbolic, and hence in particular isotropic.  \end{proof}
\sm

Our proof of Springer's Theorem also uses the following Lemma~\ref{lem_PR}, a variation of  \cite[Prop.~1.1]{panin-rehmann}.

\begin{lem}\label{lem_PR} Let $k$ be a field, let  $(V,q)$ be an isotropic quadratic $k$--space of dimension $r \geq 3$, and let $P=P(X)\in k[X]$ be a monic polynomial
of degree $d\ge 1$. Then there exists $v=v(X) \in V \otimes_k k[X]$ 
satisfying the following conditions:

\begin{enumerate}[label={\rm (\roman*)}]
 \item   \label{lem_PR_i} $q( v(X))\in k[X]$ is a  polynomial of  degree  $2d-2$,  which is divisible by $P$;

  \item \label{lem_PR_ii}  the $k$--scheme $Z_{v(X)}= \{ x  \in \GG_a : v(x)=0 \} $ is empty.
\end{enumerate}
In particular, $v(X)$ is unimodular.
\end{lem}

\begin{proof} Since $q$ is isotropic, $(V,q)$ contains
a hyperbolic plane $\HH$ and $(V,q) = \HH \perp (W,q|_W)$ for $(W, q|_W) = \HH^\perp$, see \ref{isotrop}\eqref{isotrop-c} or \cite[7.13]{EKM}. The quadratic module $(W,q|_W)$ contains $w\in W$ with $q(w) =:a \in k\ti$. In view of our claims, it is then no harm to replace $V$ by $\HH \oplus ka$. Thus $q$ is given by  $q(x,y,z)=xy + a z^2$, which is nonsingular by \ref{qfba-sing}\eqref{qfba-sing-vii} and \ref{qfba-sing}\eqref{qfba-one} (but not regular in characteristic $2$).

Since $P$ is monic, the Euclidean division algorithm yields  unique polynomials $Q(X)\allowbreak  \in k[X]$ of degree $d-2$ and $R(X)$ of degree $\leq d-1$ such that $X^{2d-2}=
P(X) Q(X) + R(X)$.  We define $v(X)= ( -a , R(X),  X^{d-1} )\in V\ot k[X]$. Then
 $q(v(X))= -a R(X)+ a X^{2d-2}= a P(X) Q(X)$. Thus, the  condition \ref{lem_PR_i} is fulfilled with $q(v(X))= a P(X) Q(X)$. Since the first component of $v(X)$ is $-a$,  the condition \ref{lem_PR_ii} is satisfied too. It implies unimodularity of $v(X)$ by \eqref{isotrop-suff1}.  \end{proof}

\subsection{Consequences of Lemma~\ref{lem_PR}.}\label{lem_PR-inter}
As motivation for step~\eqref{thm_springer-IV} in the proof of  Theorem~\ref{thm_springer} below, we discuss some consequences of Lemma~\ref{lem_PR}. Step~\eqref{thm_springer-IV} will be more technical, but uses the same ideas. Let us put
\[ S = k[X]/(P), \quad \theta = X + (P) \in S.\]
Then $S$ is one-generated with $\theta $ as primitive element. The element
$v(\theta) \in (V \ot_k k[X])\ot_{k[X]} S = V\ot_k S$ is unimodular since it is obtained from the unimodular $v(X)$ by base change, see \ref{isotrop}\eqref{isotrop-a}, and it satisfies $q\big(v(\theta)\big) = 0$ by  condition~\ref{lem_PR_i} of \ref{lem_PR}, i.e., $v(\theta)$ is an isotropic vector in $V\ot_k S$.

Again by condition~\ref{lem_PR_i} of \ref{lem_PR}, the exists $u\in k\ti$ and $Q(X) \in k[X]$ of degree $d-2$ such that $q\big(v(X)\big) = u P(X) Q(X) \in k[X]$. Let
\[ T = k[X]/(Q), \quad \vth = X + (Q) \in T. \]
Then $T$ is one-generated of degree $d-2$ with primitive element $\vth$. The same arguments showing that $v(\theta)$ is an isotropic vector proves that $v(\vth) \in V \ot_k T$ is isotropic.

In  step~\eqref{thm_springer-IV} of the proof of  Theorem~\ref{thm_springer} we will know that $V\ot_k S$ is isotropic and use a refinement of the argument above to conclude that $V\ot_k T$ is isotropic, the point being that $\deg T = \deg S - 2$.  \ms

\subsection{Proof of Theorem~\ref{thm_springer}}\label{thm_springer-pr}
\begin{inparaenum}[(I)]
  \item {\em Reduction to $M$ free of rank $r\ge 3$.}
  Let $R= R_0 \times \cdots \times R_n$ and $(M,q) = (M_0, q_0) \perp \cdots \perp (M_n,q_n)$ be the rank decomposition of $(M,q)$ as in \ref{qfba-sing}\eqref{qfba-red}. Thus, $M_i$ is a projective $R_i$--module of rank $i$ and each $q_i \co M_i \to R_i$ is a nonsingular quadratic form. The $R$--algebra $S$ decomposes correspondingly, $S=S_0 \times \cdots \times S_n$  where each $S_i$ is a finite $R_i$--algebra of degree $d=\deg S$. We have
  \[ M \ot_R S \cong (M_0 \ot_{R_0} S_0) \times \cdots \times (M_n \ot_{R_n} S_n) \]
 with each $M_i \ot_{R_i} S_i$ being projective of rank $i$ as $S_i$--module. Since $q_{S}$ is isotropic, so is every $(q_i)_{S_i}$. By \ref{isotrop}\eqref{isotrop-c}, $M_i \ot_{R_i} S_i = 0$ for $i=0,1$. Since in both cases $(S$ one-generated or $S$ \'etale) the $R_i$--modules $S_i$ are faithfully flat, we get $M_0=0=M_1$ and $R_0=0 = R_1$.

 In the decomposition $R= R_2 \times \cdots \times R_n$, each $R_i$ is a semilocal ring.  We have already observed that $(M,q)_S$ is isotropic if and only if every $M_i \ot_{R_i} S_i$ is isotropic. Since the analogous fact holds for $(M,q)$, it suffices to prove that every $(M_i, q_i)$ is isotropic. Thus, without loss of generality, we can assume that $M$ has constant rank $r \ge 2$. The case $r=2$ has been dealt with in Lemma~\ref{ratwo}. We can therefore assume that $M$ has rank $r\ge 3$. Since $R$ is semilocal, this implies %
 that $M$ is free of rank $r$.
\sm

\item \label{thm_springer-III} {\em $R=k$ is a field and $S$ is one--generated.}
In this case $S\cong k[X]/P$ for some monic $P\in k[X]$. Let $P= P_1^{e_1} \cdots P_n^{e_n}$ be the prime factor decomposition of $P$ in $k[X]$, and put $L_i = k[X]/P_i$. Since $d=\dim_k S = \sum_i e_i [L_i : k]$ is odd, one of the $[L_i : k]$ is odd. Then $L =L_i\in \Salg$ and thus $q_L = (q_S) \ot_S L$ is isotropic. Since
$L/k$ has odd degree, the classical Springer Odd Degree Extension Theorem \cite[Cor.~18.5]{EKM} says that $q$ is isotropic.
\sm

\item \label{thm_springer-IV} {\em $R$ semilocal and $S$ is one-generated.} As before, let $S=R[X]/(P)$ where $P \in R[X]$ is a monic polynomial of degree $d$. We denote by $\kappa_1, \dots, \kappa_c$ the residue fields  of $R$. For each $i$, $1\le i \le c$, the $\kappa_i$--algebra $S_{\ka_i} = S \otimes_R \kappa_i$ is one-generated, namely $S_{\ka_i} = \ka_i[X]/P_{\ka_i}$ for $P_{\ka_i} = P \ot_R 1_{\ka_i}$, and of odd degree $d$. Since $q$ is $S$--isotropic, it is also $S \otimes_R \kappa_i$-isotropic. The case~\eqref{thm_springer-III} then shows that $q_{\kappa_i}$ is isotropic.
    Now Lemma \ref{lem_PR} provides  unimodular elements $v_i(X) \in (M \ot_R \ka_i)\ot_{\ka_i} \ka_i[X] = M\otimes_R \kappa_i[X]$ such that $q( v_i(X))$ is the product of a unit in $\kappa_i^\times$ and a monic polynomial of  degree $2d-2$ which is divisible by $P_{\kappa_i}$    and has the property that the $\ka_i$--scheme
\begin{equation*}\label{cond_unimod}
     Z_{v_i} = \{ x  \in \GG_{a, \kappa_i} :  v_i(x)=0 \} \text{ is empty.}
\end{equation*}
    Let  $\theta=X + (P) \in S$ and denote by $\theta_i$ its image in $S_{\kappa_i}$.  Then $v_i(\theta_i)$ is obtained from the unimodular $v_i(X)$ by  base change and is  therefore unimodular. It also satisfies $q( v_i(\theta_i))=0$ since $q(v_i(X))$ is divisible by $P_{\ka_i}$. In other words, $v_i(\theta_i)$ is an $S_{\ka_i}$--isotropic vector.

According to Corollary \ref{cor_qquadric} of the appendix,
the $v_i(\theta_i)'s \in M \otimes_R S_{\kappa_i}$ lift
to an isotropic $v \in M \otimes_R S $. We decompose $v= m_0 +  m_1 \theta + \cdots +  m_{d-1} \theta^{d-1}$ where the $m_j$'s belong to $M$, and define $v(X)= m_0 +   m_1 X + \cdots + m_{d-1} X^{d-1} \in M\ot_R R[X]$. By construction $q(v(X)) \in P(X) R[X]$ is a polynomial of degree $\leq 2 d-2$. Since the specialization to each $\kappa_i[X]$ is of degree $2 d-2$, it follows that  $q(v(X))$ is the product of a unit $u \in R^\times$ and a monic polynomial of degree $2d-2$. Summarizing, we have  that $q(v(X)) = u \, P(X)\,  Q(X)$ with $Q(X)$ monic of degree $d-2$ and $u \in R^\times$.

Since $Z_{v_i} = \{ t  \in \GG_{a, \kappa_i} :  v_i(t)=0 \} $
is empty for each $i$, it follows that  $Z_v = \{ x  \in \GG_{a, R} : v(x)=0 \} $ is empty too. We define $T= R[X]/Q(X)$. Then $w= v(X)$ modulo $Q$ is an
isotropic vector of $M_T$. Thus, $q$ is $T$-isotropic with
$T$ one-generated of degree $d-2$. We continue the induction until we reach
$d-2=1$ and can conclude that $M$ is isotropic. \sm

\item {\em $S$ is \'etale.} According to \cite[Prop.~7.3]{BFP} there exists a finite \'etale $R$--algebra $T$ of constant odd degree such that $S\ot_R T$ is one-generated as $R$--algebra. The paper \cite{BFP} assumes throughout that $2\in R\ti$, but the proof of the quoted proposition works for arbitrary $R$.

Since $q$ is $S$-isotropic, it is a fortiori  $S \otimes_R T$-isotropic. Since $S\ot_R T$ has constant odd degree, the preceding case \eqref{thm_springer-IV} shows that $q$ is isotropic. \end{inparaenum} \hfill{\qed}%
\ms

As in the case of fields, see e.g.\ \cite[II]{Sc}, Theorem~\ref{thm_springer} has a number of consequences worth stating. First, by \ref{isotrop}\eqref{isotrop-c}, Springer's Theorem says: if $(M,q)_S$ contains a hyperbolic plane $\HH$, then so does $(M,q)$. Corollary~\ref{cor-springer} says that this is true for arbitrary hyperbolic spaces. We say {\em a quadratic module $(M,q)$ contains a quadratic module $(M_1, q_1)$\/} if there exists a complemented submodule $N\subset M$ such that
$(M_1,q_1) \cong (N,q|_N)$. In this case, we usually identify $(M_1, q_1)=(N,q|_N)$.  We recall that if $(M_1, q_1)$ is regular, e.g., a hyperbolic space, then
$(M, q) = (M_1, q_1) \perp (M_1, q_1)^\perp$ by \ref{qfba}\eqref{qfba-or}.

\begin{cor}\label{cor-springer} Let $R$, $S$ and $(M,q)$ be as in Theorem~{\rm \ref{thm_springer}}. If $(M,q)_S$ contains a hyperbolic space $\HH(N')$ with $N'$ projective of constant rank $r$, then $(M,q)$ contains a hyperbolic space $\HH(N)$ with $N$ projective of rank $r$.
\end{cor}

\begin{cor} \label{cor-springer2} Let $R$, $S$ and $(M,q)$ be as in Theorem\/~{\rm \ref{thm_springer}}, let $(M_1, q_1)$ be a regular quadratic $R$--module such that  $(M_1, q_1)_S$ is contained in $(M,q)_S$. Then $(M_1, q_1)$ is contained in $(M,q)$. \end{cor}

\begin{proof} We apply the rank decomposition~\ref{qfba-sing}\eqref{qfba-red} to $(M_1, q_1)$ and can then without loss of generality assume that $(M_1, q_1)$ has constant rank. Since $(M_1, q_1)_S$ is regular, the nonsingular quadratic $S$--space $(M,q)_S$ decomposes,
\[
    (M,q)_S \cong (M_1, q_1)_S \perp (M'_2, q'_2)
 \]
 where $(M'_2, q'_2)$ is a nonsingular $S$--space by \ref{qfba-sing}\eqref{qfba-sing-vii}. By Corollary~\ref{carhyp-cor}, the quadratic $S$--module $q_S \perp (-q_1)_S$ contains a hyperbolic $S$--space isometric to $\HH(M_{1,S})$.  Hence by Corollary~\ref{cor-springer}, there exists a nonsingular quadratic $R$--module $(M_2, q_2)$ such that
\[
 q\perp (-q_1) \cong \HH(M_1) \perp q_2 \cong q_1 \perp (-q_1) \perp q_2.
\]
Canceling $-q_1$ by Witt cancellation~\ref{qfba-sing}\eqref{qfba-wi}, yields the result.
\end{proof}

For easier reference, we explicitly state the case $(M_1, q)_S \cong (M_1, q_1)_S$ in the following Corollary~\ref{prop_odd}. The case~\ref{prop_odd-b} is obtained by writing $S/R$ as a finite tower of odd one-generated (= simple) field extensions.

\begin{cor} \label{prop_odd} Let $R$ be a semilocal ring, let $q$ and $q'$ be regular $R$--quadratic forms, and let $S\in \Ralg$ be a finite $R$--algebra of constant odd degree which satisfies one of the following conditions,
\begin{enumerate}[label={\rm (\roman*)}]
 \item \label{prop_odd-u} $S$ is a one-generated $R$--algebra, or

\item \label{prop_odd-a} $S$ is an \'etale $R$--algebra, or

\item \label{prop_odd-b} $R$ is a field and $S/R$ is a field extension. \end{enumerate}
Then \begin{equation*} \label{prop_odd-1}
       q_S \cong q'_S \iff q \cong q'.
\end{equation*}
\end{cor}

In Corollary~\ref{cor-spri-witt}, $\hWq(R)$ and $\Wq(R)$ denote the Witt-Grothen\-dieck ring and Witt ring of regular quadratic $R$--modules respectively, as for example defined in \cite[I, \S2, \S4]{Ba}. These are commutative associative rings without a multiplicative identity if $2\notin R\ti$. The proof of \ref{cor-spri-witt} is standard, using Corollary~\ref{cor-springer2}.

\begin{cor}  \label{cor-spri-witt}
Let $R$ be a semilocal ring and let $S\in \Ralg$ be a finite $R$--algebra of constant odd degree, which is one-generated or \'etale. Then the maps
\begin{equation}\label{cor-spri-witt1}
 \hWq(R) \longto \hWq(S) \quad \text{and} \quad \Wq(R) \longto \Wq(S),
 \end{equation}
induced by $[q]\mapsto [q_S]$, are monomorphisms.
\end{cor}

Corollary~\ref{cor-spri-witt} was established in \cite[V, Thm.6.9]{Ba} for Frobenius extensions, based on a detailed study of torsion in the kernels of the maps in \eqref{cor-spri-witt1} and in this way avoiding Springer's Theorem, which is not proven in \cite{Ba}.

\subsection{Set of values} \label{sov}
By definition, the {\em set of values $D(q)$\/} of a quadratic module $(M,q)$ is
\[ \rmD(q) = q(M) \cap R\ti. \]
We will use the following elementary facts. \ms

\begin{inparaenum}[(a)] \item\label{sov-a} The set of values is stable under base change: if $S\in \Ralg$, then $\rmD(q) \ot 1_S \subset \rmD(q_S)$. \sm

\item \label{sov-c} If $q$ contains a hyperbolic plane, then $\rmD(q) = R\ti$.  Recall from \ref{isotrop}\eqref{isotrop-c} that $q$ contains a hyperbolic plane whenever $(M,q)$ is an isotropic  quadratic space and $R$ is semilocal or $q$ is regular. \sm

\item\label{sov-d} ({\em Direct products}) Let $R=R_1 \times R_2$ be a direct product. By \ref{qfba-sing}\eqref{qfba-red}, the quadratic module $(M,q)$ uniquely decomposes as $(M,q) = (M_1, q_1) \times (M_2, q_2)$ where $(M_i, q_i)$, $i=1,2$, is a quadratic $R_i$--module, which is nonsingular if $(M,q)$ is so. We have     $\rmD(q) = \rmD(q_1) \times \rmD(q_2)$.
    An $S\in \Ralg$ which is projective of rank $d\in \NN_+$ uniquely decomposes as $S=S_1 \times S_2$ where each $S_i$ is a projective $R_i$--module of rank $d$.
    In view of \ref{qfba-sing}\eqref{qfba-red}, this shows that the determination of $\rmD(q)$ can often be reduced to that of $\rmD(q)$ where $(M,q)$ has constant rank. \sm

\item \label{sov-e}Let $(M,q) = (R, \lan u \ran)$ with $u\in R\ti$. Then $\rmD(q) = u R\ti{}^2$.
For any $S\in \Ralg$ which is projective of constant odd rank $d$ we have
\begin{equation}\label{solv-e1}
  a\ot 1_S \in  \rmD(q_S) \quad \iff \quad a \in \rmD(q).
  \end{equation}
By \eqref{sov-a} we only need to prove ``$\implies$''. We have $\rmD(q_S) = (u \ot 1_S )\, S\ti{}^2$, and if $a\ot 1_S = (u \ot 1_S)s^2$ for some $s\in S$, then $a^d = \rmN_{S/R}(a\ot 1_S) = u^d\, \rmN_{S/R}(s)^2$. Since $a^d \in a R\ti{}^2$ and $u^d \in u R\ti{}^2$, we get $a\in uR\ti{}^2 = \rmD(q)$.
\sm

In the following Corollary~\ref{cor-spri-rep} we will prove \eqref{solv-e1} for more general nonsingular forms.
\end{inparaenum}

\begin{cor} \label{cor-spri-rep} Let $R$ be a semilocal ring, let $a\in R\ti$ and let $(M,q)$ be a quadratic space for which $q'=q \perp \lan -a \ran $ is nonsingular, cf.\ {\rm \ref{qfba-sing}\eqref{qfba-sing-vii}}. Furthermore, let $S\in \Ralg$ be a finite $R$--algebra of constant odd degree which is \'etale or one-generated. Then
\begin{equation}\label{cor-spri-rep1} a\ot 1_S \in  \rmD(q_S) \quad \iff \quad a \in \rmD(q).
  \end{equation}
\end{cor}

This is a well-known result in case $R$ is a field of characteristic $\ne 2$ and $S/R$ is an extension field, see for example \cite[VII, Cor.~2.9]{Lam-qf}. In this case, no assumption on $q'$ is necessary.

\begin{proof}
  We will of course only prove ``$\implies$''. By \ref{sov}\eqref{sov-d} and \ref{sov}\eqref{sov-e} we can assume that $M$ has constant rank $\ge 2$. The assumption implies that $q'_S$ contains an isotropic vector $(x,1)$ for some $x\in M_S$. Hence, by Theorem~\ref{thm_springer} for $q'$, we get that $q'$ is $R$--isotropic, i.e., there exists $m\in M$ and $r\in R$ such that $(m,r) \in M \oplus R$ is unimodular and satisfies $q(m)= a r^2$. At this point, two cases are clear:

\begin{enumerate}[label={\rm (\roman*)}]
   \item $r=0$: Then $m\in M$ is an isotropic vector of $q$. Hence, by \ref{isotrop}\eqref{isotrop-c}, $(M,q)$ contains a hyperbolic plane and then we are done by \ref{sov}\eqref{sov-c}.

   \item \label{cor-spri-rep-ii} $r\in R\ti$: Then $a=q(r\me m)$ and we are again done.
  \end{enumerate}

\noindent In particular, \eqref{cor-spri-rep1} holds in case $R$ is a field. For general $R$ we will choose an isotropic vector of $q'$ more carefully.

Let $\m \ideal R$ be a maximal ideal of $R$ with residue field $\ka = R/\m$, and put $S\ot_R \ka = S /\m S$. We thus have extensions
\[ \vcenter{\xymatrix@C=40pt{S \ar[r] & S\ot_R \ka \\ R \ar[u] \ar[r] & \ka \ar[u]}
} \; .\]
Since nonsingularity is inherited by extensions and since $(S\ot_R \ka)/\ka$ has odd degree and is \'etale or one-generated, the field case applies and yields the existence of $m_\ka \in M\ot_R \ka$ satisfying $q_\ka(m_\ka) = a\ot 1_\ka$. In this way we obtain a family $(v_\m)_{\m\in \Specmax(R)}$ of isotropic vectors $v_\m = (m_\ka, 1_\ka) \in (M \oplus R)_\ka$. By Corollary~\ref{cor_qquadric}, we then get an isotropic element $v=(x,r)\in M \oplus R$ that lifts the $v_\m$'s. In particular, $r_\ka = 1_\ka$ for every residue field $\ka$ of $R$. This implies $r\in R\ti$. So we are done by \ref{cor-spri-rep-ii} above. \end{proof}

\appendix
\section{Lifting isotropic elements}\label{app:lifting}

The goal of this appendix is Corollary~\ref{cor_qquadric}, which gives a criterion for lifting isotropic elements from localizations. We obtain this as a consequence of a surjectivity result for localizations of quadrics (Proposition~\ref{prop_qquadric}), which in turn is a special case of Demazure's fundamental Conjugacy Theorem, proven in \cite[XXVI]{SGA3} and re-stated below.

If $C\in \Ralg$ and $\m \ideal R$ is a maximal ideal of $R$ we abbreviate $C/\m := C/\m C$. We use the terms {\em reductive (semisimple) group scheme\/} and {\em parabolic subgroup scheme\/} as defined in \cite[Tome III]{SGA3}, but refer to a group scheme over $\Spec(R)$ as an $R$--group scheme. We furthermore use the setting and the results of \cite[XXVI.3]{SGA3} for schemes of parabolic subgroup schemes.

\begin{thm}[\bf Demazure's Conjugacy Theorem] \label{thm_conj_demazure}
Let $R$ be a semilocal ring and let $\uG$ be a reductive $R$-group scheme. Denote by $\Dyn(\uG)$ its Dynkin $R$-scheme (which is finite \'etale) and by $\Of(\Dyn(\uG))$ the $R$-scheme of clopen subsets of $\Dyn(\uG)$. Let $t \in \Of(\Dyn(\uG))(R)$ be a type of parabolic subgroups and denote by  $\uX= \mathrm{Par}(\uG)_t$ the $R$-scheme
of parabolic subgroups of type $t$. Then the following hold.
\sm
\begin{enumerate}[label={\rm (\alph*)}]

\item \label{thm_conj_demazure-a} $\uG(R)$ acts transitively on $\uX(R)$.
\sm

\item \label{thm_conj_demazure-b} If $S$ is a
finite $R$-algebra such that $\uX(S) \not = \emptyset$, the map
\[ \uX(S) \longto \prod\limits_{\m \in \Specmax(S)} \uX(S/ \m) \]
is onto.
\end{enumerate}
\end{thm}

\begin{proof} \ref{thm_conj_demazure-a}
If $\uX(R)=\emptyset$, the statement is obvious. We can thus assume that $\uX(R) \ne \emptyset$ and pick a point $x \in \uX(R)$; this corresponds  to an $R$-parabolic subgroup $\uP$ of $\uG$ of type $t$. According to \cite[XXVI.4.3.5]{SGA3},
$\uG$ admits a parabolic subgroup $\uP'$ opposite to $\uP$.
Corollary 5.2.\ of {\em loc.\ cit.} then says in particular that $\uX(R)=\uG(R)/\uP(R)$. Thus $\uG(R)$ acts transitively on $\uX(R)$.%
\sm

\ref{thm_conj_demazure-b} Recall that  $S$ is semilocal, for example by \cite[VI, (1.1.1)]{K}. Our assumption is that $\uG_S$ admits an $S$-parabolic subgroup scheme $\uQ$ of type $t$. As noted in \ref{thm_conj_demazure-a}, it admits an opposite  parabolic $S$-subgroup $\uQ'$.
According to  \cite[XXVI.5.2]{SGA3}, the product
$\rad_u(\uQ)(S)  \times \rad_u(\uQ')(S) \to \uX(S)$
is surjective (here $\rad_u(.)$ denotes the unipotent radical).
Applying this for the semilocal ring $S$ as well as
for the semilocal ring $S/\m$, $\m \in \Specmax(R)$,
shows that the horizontal maps in the
commutative diagram below are surjective:
\[ \xymatrix{
    \rad_u(\uQ)(S)  \times \rad_u(\uQ')(S)  \ar[r] \ar[d]
    &  \uX(S) \ar[d] \\
     \prod\limits_{\m} \rad_u(\uQ)(S/\m)  \times \rad_u(\uQ')(S/\m)
     \ar[r]
    &  \prod\limits_{\m} \uX(S/\m)
} \]
Since $\prod_\m S/\m \cong \prod_\m S \ot_R (R/\m) \cong S \ot_R \big(R/\Jac(R)\big) \cong S/\Jac(R)S$, where $\Jac(R)$ denotes the Jacobson radical of $R$, the map $S \to \prod_{\m} S/\m$ is onto. On the other hand, the $S$--scheme $\rad_u(\uQ)$ (respectively $\rad_u(\uQ')$) is isomorphic to a vector $S$--group scheme \cite[XXVI.2.5]{SGA3}, so that  the left vertical map is onto. Hence,
by a simple diagram chase the right vertical map  is onto too.
\end{proof}

\begin{cor}\label{cor_conj_demazure}
  We use the notation of {\rm \ref{thm_conj_demazure}}, except that $R$ need not be semilocal, but can be arbitrary. In particular, $\uG$ is a reductive $R$--group scheme and $x$, $y$ are parabolic subgroups in $\uX(R)$.

  Then  there exist $f_1,\ldots, f_n \in R$ satisfying $f_1+\cdots+ f_n= 1$ and $y_{R_{f_i}} \in \uG(R_{f_i}) \, . \, x_{R_{f_i}}$ for $i=1,..,n$.
In other words, $x$ and $y$ are locally $\uG$-conjugated for the Zariski topology on $R$.
\end{cor}

\begin{proof} The group $\uG$ is an $R$--group of type (RR) by \cite[XXII.5.1.3]{SGA3} and the points $x$, $y$ of $\uX(R)$ are parabolic subgroups, hence subgroups of type (R) by 5.2.3 of {\em loc.\ cit.}. It then follows from Theorem~5.3.9 of {\em loc.\ cit.} that the strict transporter $\uT$, defined by
\[ \uT(S) = \{ g\in \uG(S) : g \cdot x_S = y_S\} \quad (S\in \Ralg), \]
is a finitely presented affine $R$--scheme (among other properties). Since $\uT(R_\m) \ne \emptyset$ for any maximal $\m \in \Spec(R)$ by \ref{thm_conj_demazure}\ref{thm_conj_demazure-a}, the claim follows from Lemma~\ref{fip} below. \end{proof}

\begin{lem}\label{fip} Let $R$ be arbitrary and let $\uT$ be an $R$--scheme which is locally of finite presentation. If\/ $\uT(R_\m) \ne \emptyset$ for all maximal $\m \in \Spec(R)$, then there exists a Zariski cover $(f_1, \ldots, f_n)$ of $R$ for which $\uT(R_{f_i}) \ne \emptyset$ for all $i$, $1 \le i \le n$.
\end{lem}

\begin{proof} Fix a maximal $\m\in \Spec(R)$. Then $R_\m = \limind_{f \not \in \m} R_f$, hence $\Spec(R) = \limproj_{f\not\in \m} \Spec(R_f)$ and so $
\uT(R_\m)= \limind_{f \not \in \m} \uT(R_f)$,
according to \cite[Tag 01ZC]{Stacks}. 
It follows that there exists $f_\m \in R \setminus \m $ such that $\uT(R_{f_\m}) \not=\emptyset$. The $f_\m$'s for $\m$ running over the maximal ideals of $R$ generate $R$ as ideal. Hence, there exist finitely many maximal ideals $\m_1, \dots, \m_n$ such that $R=R f_{\m_1}+ \dots +  R f_{\m_n}$. \end{proof}

\subsection{}\label{prop_qquadric-not} An important special case of Theorem~\ref{thm_conj_demazure} and Corollary~\ref{cor_conj_demazure} is that of quadrics elaborated below. Let us explain some notation. We let $(M,q)$ be a quadratic space with $M$ of constant rank $n$. The associated special orthogonal group scheme $\uG=\uSO(q)$ is defined in \cite[C.2.10]{Co1}. It is a semisimple $R$--group scheme of type $\rmA_1$ for $n=3$, of type $\rmB_{(n-1)/2}$ for odd $n\ge 5$ and of type $\rmD_{n/2}$ for even $n\ge 4$. We use Bourbaki's enumeration of the corresponding Dynkin diagrams:
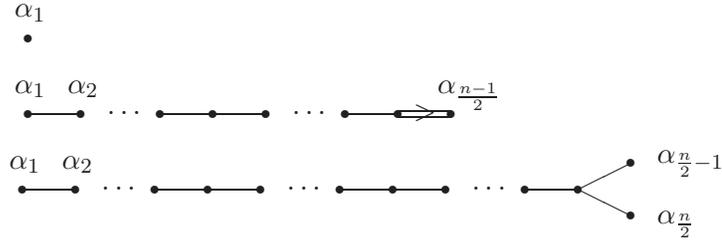
\begin{align*} 
&\begin{picture}(200,20)
\put(30,0){\circle*{3}}
\put(25,8){$\alpha_1$}
\end{picture} 
\\ 
&\begin{picture}(200,20)
\put(30,00){\line(1,0){20}}
\put(60,00){\dots}
\put(80,00){\line(1,0){40}}
\put(130,00){\dots}
\put(150,00){\line(1,0){20}}
\put(170,1.1){\line(1,0){20}}
\put(170,-1.2){\line(1,0){20}}
\put(176,-2.5){$>$}
\put(30,0){\circle*{3}}
\put(50,0){\circle*{3}}
\put(80,0){\circle*{3}}
\put(100,0){\circle*{3}}
\put(120,0){\circle*{3}}
\put(150,0){\circle*{3}}
\put(170,0){\circle*{3}}
\put(190,0){\circle*{3}}
\put(25,8){$\alpha_1$}
\put(45,8){$\alpha_2$}
\put(95,8){}
\put(160,8){}
\put(185,8){$\alpha_{\frac{n-1}{2}}$}
\end{picture}
\\
& \;\; \qquad \begin{picture}(250,20)
\put(00,00){\line(1,0){20}}
\put(30,00){\dots}
\put(50,00){\line(1,0){40}}
\put(100,00){\dots}
\put(120,00){\line(1,0){40}}
\put(170,00){\dots}
\put(190,00){\line(1,0){20}}
\put(210,0){\line(2,-1){20}}
\put(210,0){\line(2,1){20}}
\put(0,0){\circle*{3}}
\put(20,0){\circle*{3}}
\put(50,0){\circle*{3}}
\put(70,0){\circle*{3}}
\put(90,0){\circle*{3}}
\put(120,0){\circle*{3}}
\put(140,0){\circle*{3}}
\put(160,0){\circle*{3}}
\put(190,0){\circle*{3}}
\put(210,0){\circle*{3}}
\put(230,10){\circle*{3}}
\put(230,-10){\circle*{3}}
\put(-5,8){$\alpha_1$}
\put(15,8){$\alpha_2$}
\put(65,8){}
\put(127,8){}
\put(240,10){$\alpha_{\frac{n}{2}-1}$}
\put(240,-13){$\alpha_{\frac{n}{2}}$}
\end{picture}
\end{align*}
\ms

\noindent We also use the projective space $\PP(M^\vee)$ with Grothendieck's convention, i.e., $\PP(M^\vee)(S)$, $S\in \Ralg$, corresponds to the complemented submodules  $D$ of $M_S$ which are locally free of rank $1$. The quadric $\uQ$ defined by $q=0$ consists of those $D$ in $\PP(M^\vee)(S)$ with $q(D) = 0$. For $D \in \uQ(R)$ we let $\uP$ be the $R$-subgroup scheme of  $\uG$ which stabilizes $D$.
Finally, we remind the reader of our abbreviation $S/\m = S/\m S$ for $S\in \Ralg$ and $\m \in \Specmax(R)$.

\begin{prop}\label{prop_qquadric} We use the notation of {\rm \ref{prop_qquadric-not}}.  Then the following hold. \sm

\begin{enumerate}[label={\rm (\alph*)}]
 \item \label{prop_qquadric1}
\begin{enumerate}[label={\rm (\roman*)}]
  \item \label{prop_qquadric1-i}
 $\uP$ is a parabolic $R$--subgroup of $\uG$.

 \item  \label{prop_qquadric1-ii} The  orbit map $\uG \to \PP(M^\vee)$, $g \to g.[D]$, induces  an isomorphism $\uG/\uP \simlgr \uQ$ of $R$--schemes.

\item \label{prop_qquadric2}
$\uQ$ is $\uG$-isomorphic to $\mathrm{Par}(\uG)_{t_1}$
where the type $t_1$ is constant and of value $\Dyn(G)(R) \setminus \{\al_1\}_R$
with the enumeration of the respective Dynkin diagrams displayed above.
\end{enumerate}
\sm

\item \label{prop_qquadric3}  Let $R$ be semilocal. Then
\begin{enumerate}[label={\rm (\roman*)}]
 \item   \label{prop_qquadric3-i}  $\uG(R)$  acts transitively on $\uQ(R)$, and

 \item \label{prop_qquadric3-ii} if $S$ is a finite
$R$-algebra such that $\uQ(S) \not = \emptyset$, then
\begin{equation} \label{prop_qquadric4-1}
\uQ(S) \longto \prod\limits_{\m \in \Specmax(R)} \uQ(S/\m)
\end{equation}
is onto.
\end{enumerate}
\end{enumerate}
\end{prop}

\begin{proof}
We prove  only  the  even rank case, the odd case being analogous.
\sm

\noindent \ref{prop_qquadric1}
We observe that the claims hold over a field \cite[Th. 3.9.(i)]{Co3} and also that they are local for the flat topology. According to \cite[Lemma C.2.1]{Co1} or \cite[IV, (3.2.1)]{K}, we may  assume that $(M,q)$ is the standard hyperbolic
quadratic $R$-form over $R^{2n}$ ($n \geq 2$), defined by $q(x_1,\dots, x_n, y_1, \dots, y_n)= x_1 y_1 +  \dots + x_n y_n$.

We first consider the case $D=Rp$ where $p=(1,0, \dots, 0, 0,  \dots ,0 )$, and denote by $\uP_p$ the corresponding subgroup scheme. This permits us to  assume  $R=\ZZ$. Then $\uP_p$ is an affine finitely presented $\ZZ$-group scheme whose algebraically closed fibers are smooth connected and of dimension $2n-2$. By \cite[Lemma B.1]{AGi}, $\uP_p$ is smooth. Since the geometric fibers of
$\uP_p$ are parabolic subgroups, $\uP_p$ is parabolic too.
The induced  orbit map $f:  \uG/\uP_p \to \uQ$ is a monomorphism.
The field case ensures that this is a fiberwise  isomorphism. Since $\uG/\uP_p$ is flat and of finite presentation, the fiberwise isomorphism criterion \cite[IV$_4$, 17.9.5]{EGA} enables us to conclude that $f$ is an isomorphism. It follows that $\uQ$ is homogeneous under $\uG$. Thus \ref{prop_qquadric1-i} and \ref{prop_qquadric1-ii} hold for the special $p$.

Let us now deal with the general case. Since $\uG$-homogeneity is  a local property with  respect to the flat topology, $\uQ$ is homogeneous under $\uG$.
The case of a general $D$ then follows from the observation that $D$ is locally
$\uG$-conjugated to $Rp$ in the Zariski topology by applying  Corollary \ref{cor_conj_demazure} to $\uQ$. This proves \ref{prop_qquadric1-i} and \ref{prop_qquadric1-ii} in general.
 \sm

\noindent  \ref{prop_qquadric2} We first deal with the special $p$ above.
We  have an isomorphism $\uG/\uP \cong \uQ$. On the other hand, according to \cite[XXVI.3.6]{SGA3}, we have a $\uG$-isomorphism $\uG/\uP \simlgr \mathrm{Par}(\uG)_{t(\uP)}$,  hence an
 $\uG$-isomorphism $\uQ \simlgr \mathrm{Par}(\uG)_{t(\uP)}$.
 This isomorphism applies a point $x \in \uQ(R)$ to the stabilizer $\uG_x$, so is a canonical isomorphism.  Here $t(\uP) \in \mathrm{Of}( \mathrm{Dyn}(\uG))(R)$
 is the type of $\uP$.  Checking that it is $t_1$, reduces to the field case which is \cite[Lemma 3.12]{Co3}.

For the general case, let $S$ be a flat cover of $R$ such that $\uQ(S) \not
=\emptyset$.
We have an isomorphism $\uQ_S \simlgr \mathrm{Par}(G)_{t_1,S}$
which is canonical and $\uG_S$-equivariant.
By faithfully flat descent, it descends to a $\uG$-equivariant isomorphism
$\uQ \simlgr  \mathrm{Par}(\uG)_{t_1}$.
\sm

\noindent  \ref{prop_qquadric3} follows from Theorem~\ref{thm_conj_demazure}
applied to $\uQ \cong \mathrm{Par}(\uG)_{t_1}$.
\end{proof}

\begin{cor} \label{cor_qquadric}
We assume that $R$ is a semilocal ring. Let $(M,q)$ be a quadratic space
of constant  rank $\geq 3$, let $S$ be a finite $R$-algebra such that $q_S$ is isotropic, and let $(v_\m)_{\m \in \Specmax(R)}$ be a family of isotropic elements $v_\m \in M \otimes_R S/\m$. Then there exists an isotropic $v \in M \otimes_R S$ that lifts the $v_\m$'s.
\end{cor}

\begin{proof}
Let $\uQ \subset \PP(M^\vee)$ be the projective quadric
associated with $q$ in Proposition~\ref{prop_qquadric}. Our assumption is that $\uQ(S) \not = \emptyset$.

Each $S/\m$--module $(S/\m). v_\m$ is a complemented submodule
of $M \otimes_R S/\m$ which is free of rank one, so defines a point
$x_\m \in \uQ(S/\m)$. Proposition \ref{prop_qquadric}\ref{prop_qquadric3}\ref{prop_qquadric3-ii}
provides an element $x \in \uQ(S)$ which lifts the $x_\m$'s.
Since $S$ is semilocal, $x$ is represented by an $S$-module $D$ which is  a complemented submodule of $M\otimes_R S$ of rank one  and satisfies $q(D)=0$. We write $D=S v$ where $v$ is an $S$--unimodular element of $M \otimes_R S$ satisfying $q(v)=0$. Since
$S^\times \to \prod_\m (S/\m)^\times$ is onto, 
we can modify $v$ by an unit of $S$ to ensure that
$v$ lifts the $v_\m$'s. \end{proof}

\section{Trace for torsors}  \label{app:torEN}

\subsection{Weil restriction (\cite[7.6]{BLR}, \cite[A.5]{CGP},  \cite[I, \S1, 6.6]{DG})} \label{were} Let $S\in \Ralg$. Given an $S$--functor $Y'$, the {\em Weil restriction\/} of $Y'$ is the $R$--functor $\sfR_{S/R}(Y')$ defined by
\[ \sfR_{S/R}(Y') (A) = Y'(A \ot_R S), \quad (A\in \Ralg). \]
It is uniquely determined by the following universal property: for  every $R$--functor $X$ there exists a bijection
\begin{equation}   \label{were1}
\xi = \xi_{X,Y'} \co \Mor_{\Salg} (X_S, Y') \simlgr \Mor_{\Ralg}(X, \sfR_{S/R}(Y'))
\end{equation}
where $X_S$ is the $S$--functor obtained from $X$ by base change, thus satisfying $X_S(B) = X({_R B})$ for $B\in \Salg$.  Here and sometimes in the following we write ${_R B}$ to denote the $R$--algebra obtained from the $S$--algebra $B$ by restriction of scalars. The bijection $\xi$ is functorial in $X$ and $Y'$. It maps $g\in \Mor_{\Salg} (X_S, Y')$ to the composition
\[ X(A) \xrightarrow{X(\inc_1)} X\big({_R(A \ot_R S) }\big) \xrightarrow{g(A \ot_R S)} Y'(A\ot_R S)
\]
where $A\in \Ralg$ and $\inc_1$ is the $R$--algebra homomorphism
\[ \inc_1 = \inc_{1,A} \co A \longto {_R(A \ot_R S)} , \quad a \mapsto a \ot 1_S.
\]

We consider two special cases of \eqref{were1}. First, for $Y' = X_S$ and $g= \Id_{X_S}$ we get the morphism
\[
 \sfj =  \sfj_X = \xi(\Id_{X_S}) \co X \longto \sfR_{S/R}(X_S) \]
of $R$--functors, determined by $\sfj_X(A) = X(\inc_{1,A})$, $A\in \Ralg$. Second, putting $X = \sfR_{S/R}(Y')$ in \eqref{were1}, there exists a unique morphism
\[ \sfq = \sfq_{Y'} \co \sfR_{S/R}(Y')_S \longto Y' \]
of $S$--functors satisfying $\xi(\sfq_{Y'}) = \Id_{\sfR_{S/R}(Y')}$. For $B\in \Salg$ we have $\sfR_{S/R}(Y')_S(B) = \sfR_{S/R}(Y')({_R B}) = Y'({_RB} \ot_R S)$  and so
\[ \sfq_{Y'}(B)  \co Y'({_R B}\ot_R S) \longto Y'(B).
\] In fact,
$\sfq_{Y'}(B) =  Y'(m_B)$ where $m_B$ is the $S$--algebra homomorphism
\[ m_B \co ({_R B}) \ot_R S  \longto B, \quad b \ot s \mapsto bs\]
(\cite[A.5.7]{CGP}). For $Y' = X_S$ we now have constructed morphisms
\[
  X_S \xrightarrow{\, \sfj_{X,S}\, } \sfR_{S/R}(X_S)_S \xrightarrow{\, \sfq_{X_S}\, } X_S
\]
 between $S$--functors. Untangling the constructions above, we find
\begin{equation}
  \label{were2} \sfq_{X_S} \circ \sfj_{X,S} = \Id_{X_S}
\end{equation}
 because $B \xrightarrow{\inc_{1,B}} {_R B} \ot_R S \xrightarrow{m_B} B$ equals $\Id_B$.

\subsection{Cohomology and restriction} \label{cores} Let $G$ be a flat $R$--group sheaf. We denote by
\[ H^1(R, G) = H^1\fppf(R, G) \]
the pointed set of isomorphism classes of $G$--torsors over $\Spec(R)$ in the flat topology. Let $S\in \Ralg$. The base change $X_S= X \times_{\Spec(R)} \Spec(S)$ of a $G$--torsor $X$ is a $G_S$--torsor, giving rise to the {\em restriction map\/}
\[ \res = \res_{S/R, \, G} \co H^1(R, G) \to H^1(S, G_S), \quad [X] \mapsto [X_S].
\]
A homomorphism $f \co G \to H$ of flat $R$--group sheaves induces a map in cohomology
\[ f_* \co H^1(R, G) \to H^1(R, H) , \quad [X] \mapsto [X \we^G H] \]
where $X\we^G H = (X \times_{\Spec(R)} H) / G$ is the contracted product with respect to the $G$--action on $H$ via $f$.
Contracted products are special cases of fppf quotients and as such allow base change \cite[(4.30)]{moonen}. Hence
$(X \we^G H)_S $ and $X_S \we^{G_S} H_S$ are isomorphic $H_S$--torsors, giving rise to a commutative diagram
\begin{equation}  \label{cores1} \vcenter{
 \xymatrix@C=40pt
 {  H^1(R, G) \ar[r]^{f_*} \ar[d]_\res & H^1(R, H) \ar[d]^\res \\
     H^1(S, G_S) \ar[r]^{f_{S,*}} & H^1(S, H_S)   }
}\end{equation}
of pointed sets. It is known that $H =\sfR_{S/R}(G_S)$ is again a flat $R$--group sheaf if $G$ is so and that the maps
\[ \sfj \co G \to \sfR_{S/R}(G_S) \quad \text{and} \quad
     \sfq\co \sfR_{S/R}(G_S)_S \to G_S
 \]
of \ref{were} are homomorphisms of $R$--group sheaves.
Hence, passing to cohomology,  the maps in the diagram are well-defined:
\begin{equation}  \label{cores2} \vcenter{
 \xymatrix@C=40pt
 {  H^1(R, G) \ar[r]^{\sfj_*} \ar[d]_{\res_{\,G}}  &
         H^1(R, \sfR_{S/R}(G_S) ) \ar[d]^{\res_{\, \sfR(G_S)}}
     \ar@{-->}[dl]_{\rmi}             \\
   H^1(S, G_S) & H^1(S, \sfR_{S/R}(G_S)_S) \ar[l]^{\sfq_*}   }
}\end{equation}
We claim that \eqref{cores2} is a commutative diagram. Indeed, since $\sfq \circ \sfj_S = \Id_{G_S}$ by \eqref{were2}, this follows from commutativity of \eqref{cores1}:
\begin{equation}  \label{cores3}
   \res_{\, G} = \sfq_* \circ (\sfj_S)_* \circ \res_{S/R, G} = \sfq_* \circ
      \res_{\, \sfR_{S/R}(G_S)} \circ \sfj_* = \rmi \circ \sfj_*.
\end{equation}
The map $\rmi = \sfq_* \circ \res_{\, \sfR_{S/R}(G_S)}$ is injective by  \cite[XXIV.8.2]{SGA3}.
\sm

In particular, {\em assume that $G$ is an abelian affine $R$--group scheme}. Then so are $G_S$, $\sfR_{S/R}(G_S)$ and $\sfR_{S/R}(G_S)_S$. 
 Moreover, the cohomology sets
and maps used in \eqref{cores2} are abelian groups and group homomorphisms respectively. Since $\rmi$ is injective, we get from \eqref{cores3} that
\begin{equation}\label{scores4}
   \Ker( \res_{\, G}) =  \Ker( \rmi \circ \sfj_*) = \Ker (\sfj_*).
\end{equation}

\subsection{Deligne trace homomorphism}\label{del-tr} Let $S\in \Ralg$ be locally free of finite rank $d\in \NN_+$, i.e., the $R$--module $S$ is projective of constant rank $d$, and let $G$ be an abelian affine $R$--group scheme. We then have Deligne's trace homomorphism
\begin{equation} \label{del-tr1} \sftr \co \sfR_{S/R}(G) \to G, \quad
  \text{satisfying $\sftr \circ \sfj = \times d$},
 \end{equation}
where   $\times d \co G \to G$ is the group homomorphism given on $T$--points, $T\in \Ralg$, by $g \mapsto g^d$ \cite[XVII, 6.3.13--6.3.15]{SGA4}. It induces an endomorphism
\[ (\times d)_* \co H^1(R, G) \longto H^1(R, G)
\]
given by the analogous formula. Since $(\times d)_* = \sftr_* \circ \sfj_*$ by \eqref{del-tr1}, we obtain from \eqref{scores4} that
\begin{equation}\label{del-tr-2}
    \Ker(\res_{S/R, G}) = \Ker(\sfj_* ) \subset \Ker\big((\times d)_*\big).
 \end{equation}
In particular, this implies the following.

\begin{lem} \label{lem_ttrace} In the setting of\/ {\rm \ref{del-tr}} assume that
$\times d: G \to G$ is an isomorphism. Then the restriction homomorphism
  $\res_{S/R, \, G} \co H^1\fppf(R,G) \to H^1\fppf(S,G)$ is injective.
\end{lem}
\sm

\textbf{Example.} Let $G$ be the constant $R$--group scheme $\ZZ/2\ZZ$, which is the automorphism group scheme of a quadratic \'etale $R$--algebra $A$, and let $d$ be odd. Then Lemma~\ref{lem_ttrace} applies and in particular shows that $A$ is split, i.e., $A=R\times R$, if and only if $A_S$ is split.

\ms

{\em Acknowledgments.} We thank M.~Ojanguren and the referee for helpful comments on an earlier version of this paper.

\end{document}